\documentclass{amsart}

\usepackage{comment}
\usepackage{graphicx,xcolor}

\numberwithin{equation}{section}
\usepackage[initials,nobysame]{amsrefs}
\usepackage{amsmath,mathtools}
\PassOptionsToPackage{hyphens}{url}\usepackage{hyperref}
\usepackage{cleveref}
\usepackage{enumitem}

\setlist[enumerate,1]{label=\upshape{(\roman*)},ref=\roman*}

\usepackage{autonum}

\newtheorem{theorem}{Theorem}[section]
\newtheorem{proposition}[theorem]{Proposition}
\newtheorem{corollary}[theorem]{Corollary}
\newtheorem{lemma}[theorem]{Lemma}
\newtheorem{conjecture}[theorem]{Conjecture}
\theoremstyle{definition}
\newtheorem{definition}[theorem]{Definition}

\newtheorem*{acknowledgements}{Acknowledgements}
\theoremstyle{remark}
\newtheorem{remark}[theorem]{Remark}

\DeclareMathOperator{\sech}{sech}

\newcommand{\R}{\mathbf{R}}
\renewcommand{\S}{\mathbf{S}}
\newcommand{\N}{\mathbf{N}}
\newcommand{\Z}{\mathbf{Z}}

\title[Embeddedness and graphicality of the elastic flow]{Embeddedness and graphicality of the elastic flow for complete curves}

\author[T.~Miura]{Tatsuya Miura}
\address[T.~Miura]{Department of Mathematics, Graduate School of Science, Kyoto University, Kitashirakawa Oiwake-cho, Sakyo-ku, Kyoto 606-8502, Japan}
\email{tatsuya.miura@math.kyoto-u.ac.jp}

\author[F.~Rupp]{Fabian Rupp}
\address[F.~Rupp]{Faculty of Mathematics, University of Vienna, Oskar-Morgenstern-Platz 1, 1090 Vienna, Austria.}
\email{fabian.rupp@univie.ac.at}

\date{\today}

\begin{document}

\begin{abstract}
We study positivity-preserving properties for the elastic flow of non-compact, complete curves in Euclidean space. Despite the fact that the canonical elastic energy is infinite in this context, we extend our recent work based on the adapted elastic energy to derive nontrivial optimal thresholds for maintaining planar embeddedness and graphicality, respectively. We also obtain a new Li--Yau type inequality for complete planar curves.
\end{abstract}

\maketitle
% \setcounter{tocdepth}{2}
% \tableofcontents

\section{Introduction}

\emph{Maximum principles} are key ingredients in the analysis of second order partial differential equations. Indeed, in the realm of geometric flows (e.g., mean curvature-type flows), they imply the preservation in time of various geometric ``positivity'' properties of the initial datum, including, for instance, convexity, embeddedness, or graphicality.
These arguments are well-known to fail in the case of \emph{higher order} geometric flows, resulting in situations where these properties are lost in finite time, see \cite{Blatt2010} for convexity and embeddedness and \cite{ElliottMaier-Paape2001} for graphicality. However, the precise extent of this phenomenon remains largely unexplored. 

Recently, several advances in this direction have been established for the \emph{elastic flow}, a typical example of fourth-order geometric flows.
The elastic flow for closed curves $\gamma$ in $\R^n$ (or non-closed curves satisfying suitable boundary conditions) is obtained as the $L^2(ds)$-gradient flow of the \emph{elastic energy}
\begin{align}
     \hat E[\gamma] \vcentcolon= B[\gamma]+L[\gamma] = \int_\gamma\left(\frac{1}{2}|\kappa|^2 +1\right) ds,
\end{align}
resulting in the fourth order evolution equation
\begin{equation}\label{eq:EF}
    \partial_t\gamma = -\nabla_s^2\kappa - \frac{1}{2}|\kappa|^2\kappa + \kappa. \tag{EF}
\end{equation}
Here $L[\gamma]=\int_\gamma ds$ denotes the \emph{length functional}, $B[\gamma]=\frac{1}{2}\int_\gamma|\kappa|^2ds$ denotes the \emph{bending energy}, $s$ is the \emph{arclength parameter}, $\kappa \vcentcolon = \partial_s^2\gamma$ is the \emph{curvature vector}, and $\nabla_s \vcentcolon = \partial_s - \langle \partial_s\gamma, \partial_s \cdot\rangle \partial_s\gamma$ is the \emph{normal derivative} along $\gamma$.
Recall that $ds=|\partial_x\gamma|dx$ and $\partial_s\vcentcolon = |\partial_x\gamma|^{-1}\partial_x$, where $x$ denotes the original variable of $\gamma$.
Initial boundary value problems involving \eqref{eq:EF} have been studied in various works, see, e.g., \cite{Mantegazza_Pluda_Pozzetta_21_survey,miura2025newenergymethodshortening} and the references therein, establishing local well-posedness, global existence, and convergence as $t\to\infty$ in various settings.

In previous work \cite{Miura_Muller_Rupp_2025_optimal}, M\"uller and the authors obtained an \emph{energy-based alternative} to maximum principles that ensures the preservation of embeddedness for the elastic flow of closed curves.
More precisely, the main finding in \cite{Miura_Muller_Rupp_2025_optimal} is a threshold $C^*(n)$ such that any embedded closed curve $\gamma\colon \S^1\to\R^n$ with $\hat E[\gamma]\leq C^*(n)$ remains embedded under the elastic flow, while any slightly larger threshold allows for an embeddedness-breaking in finite time. The constant $C^*(n)$ is determined by certain possibly non-smooth critical shapes involving \emph{elasticae}, i.e., critical points of $\hat{E}$. The value of $C^*(n)$ differs in the planar case $n=2$ from the case of higher codimension $n\geq 3$. See also \cite{MR4565935,MR4631455,Kemmochi_Miura_2024_migration,miura2024migrating} for further work on the validity and failure of positivity-preserving properties for \eqref{eq:EF}.

Here we consider non-closed complete curves (with infinite length). A fundamental issue as opposed to the closed case is that the canonical elastic energy $\hat{E}=B+L$ is \emph{always infinite}, and thus cannot provide any quantitative thresholds.
In this paper, instead, we work with the adapted elastic energy $E\vcentcolon=B+D$, where the length $L$ is replaced by the \emph{direction energy} $D[\gamma]\vcentcolon=\frac{1}{2}\int_\gamma|\partial_s\gamma-e_1|^2ds$, yielding
\begin{equation}\label{eq:energy_E=B+D}
    E[\gamma] = B[\gamma]+D[\gamma] = \int_\gamma \left( \frac{1}{2}|\kappa|^2+\frac{1}{2}|\partial_s\gamma-e_1|^2 \right) ds.
\end{equation}
Here and in the sequel $\{e_i\}_{i=1}^n$ denotes the canonical basis of $\R^n$.
This energy was first introduced and employed in \cite{Miura20,miura2024uniqueness} to variationally handle certain infinite-length elasticae.
A crucial observation is that the two different quantities $L$ and $D$ are in fact equal up to a null Lagrangian, thus yielding the same first variation. In the authors' recent work \cite{miura2025newenergymethodshortening}, this idea was extended to a class of gradient flows involving $L$, unlocking the application of energy methods in the infinite-length regime. In particular, a global well-posedness theory was established for the elastic flow \eqref{eq:EF} of complete curves, which not only improves \cite{Novaga_Okabe_2014_infinite} in several aspects but also establishes the key property that the adapted energy $E$ decays by
\begin{align}\label{eq:energy_decay}
    E[\gamma(\tau,\cdot)] + \int_0^\tau \int_{\gamma(t,\cdot)} |\partial_t \gamma|^2 ds dt = E[\gamma(0,\cdot)]\qquad \text{ for all }\tau \in [0,\infty).
\end{align}
Thus, the adapted energy $E$ opens the possibility of obtaining optimal thresholds in the spirit of \cite{Miura_Muller_Rupp_2025_optimal} for complete elastic flows.

The contribution of this work is twofold. Firstly, in analogy with \cite{Miura_Muller_Rupp_2025_optimal}, we examine embeddedness-preservation in the planar case. We identify a planar curve $\gamma_P:\R\to\R^2$ (see \Cref{def:elastic_pendant} and \Cref{fig:elastic_pendant}), which we call the \emph{elastic pendant}, that determines the optimal threshold.

\begin{figure}[htbp]
    \centering
    \includegraphics[width=0.6\linewidth]{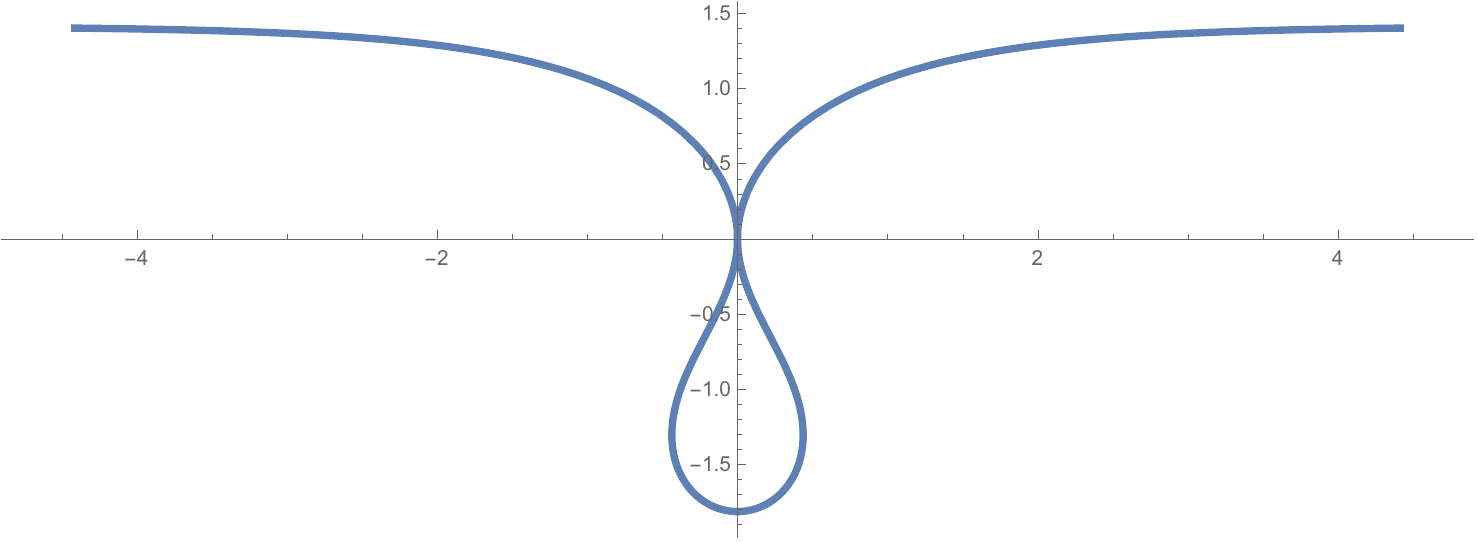}
    \caption{Elastic pendant $\gamma_P$: the optimal shape for planar embeddedness-preservation (see \Cref{def:elastic_pendant}).}
    \label{fig:elastic_pendant}
\end{figure}

The space $\dot{C}^\infty$ used in the sequel denotes smooth functions with bounded derivatives and appears in the well-posedness theory for \eqref{eq:EF} developed in \cite{miura2025newenergymethodshortening}.
\begin{theorem}\label{thm:EF_embedded}
    Let $\gamma_0\in \dot{C}^\infty(\R;\R^2)$ with $\inf_{\R}|\partial_x\gamma_0|>0$.
    Suppose that $\gamma_0$ is embedded and
    \begin{equation}
        E[\gamma_0]\leq E[\gamma_P] \quad(\approx 10.906581).
    \end{equation}
    Then the elastic flow \eqref{eq:EF} starting from $\gamma_0$ remains embedded for all $t\geq0$, and converges to a straight line as $t\to\infty$ (up to reparametrization by arclength and translation) in the sense that
    \begin{equation}
        \lim_{t\to\infty}\sup_{\R}|\partial_s\gamma-e_1|=0, \quad \lim_{t\to\infty}\sup_{\R}|\partial_s^m\kappa|=0,
    \end{equation}
    for all $m\in\N_0$.
    The energy threshold is optimal; for any $\varepsilon>0$ there exists a planar embedded initial curve $\gamma_0$ with $E[\gamma_0]< E[\gamma_P] + \varepsilon$ such that the flow is not embedded at some time $t_0>0$.
\end{theorem}

\begin{remark}
    The convergence part is also a new result, since such a convergence to a line was previously ensured only under the lower energy threshold $E[\gamma_0]<8$ \cite[Corollary 2.10]{miura2025newenergymethodshortening} (but for all $n\geq2$).
    In the planar case $n=2$, we expect the stronger property that if $E[\gamma_0]<\infty$ and $N[\gamma_0]=0$, where $N$ denotes the rotation number (see \Cref{subsec:rotation_number}), then the same convergence property holds as above.
    This may be regarded as a special case of our previous ``energy quantization'' conjecture $E[\gamma(t)]\to 8|N[\gamma_0]|$ for planar elastic flows \cite[Conjecture 7.18]{miura2025newenergymethodshortening}.
\end{remark}

\begin{remark}
    We expect that the optimal threshold for preserving embeddedness in $\R^n$ with $n\geq3$ is given by the standard borderline elastica $\gamma$ with $E[\gamma]=8$, which is different from the codimension-one case in \Cref{thm:EF_embedded}.
    A similar codimension-dependent transition is already observed in \cite{Miura_Muller_Rupp_2025_optimal} for closed curves. However, in the current situation the problem seems much more complicated.
    In fact, for closed curves, a key Li--Yau type inequality is available in all codimensions \cite{MR4631455}, but for infinite-length curves this is not the case in codimension two or higher (see \Cref{thm:borderline_minimality} in codimension one).
    See also \Cref{conj:boderline_minimality_high_codim} and subsequent discussions.
\end{remark}

Secondly, we examine the preservation of graphicality over a coordinate axis. We say that a complete curve $\gamma\colon \R\to\R^n$ is \emph{graphical} (over the $e_1$-axis) if 
\[
\inf_\R\langle\partial_s\gamma,e_1\rangle>0.
\]
In this case, the optimal shape for the threshold is given by a planar curve $\gamma_S$ (see \Cref{def:elastic_serpent} and \Cref{fig:elastic_serpent}), which we call the \emph{elastic serpent}.

\begin{figure}[htbp]
    \centering
    \includegraphics[width=0.6\linewidth]{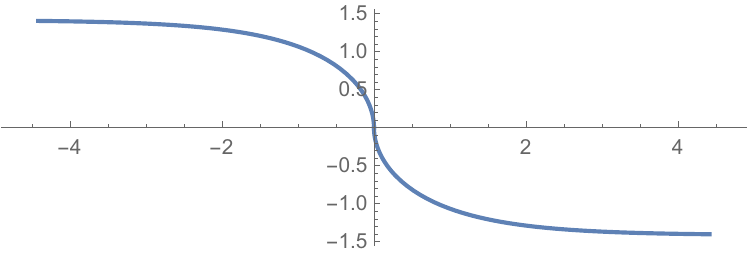}
    \caption{Elastic serpent $\gamma_S$: the optimal shape for graphicality-preservation (see \Cref{def:elastic_serpent})}
    \label{fig:elastic_serpent}
\end{figure}

\begin{theorem}\label{thm:EF_graphical}
    Let $n\geq2$ and $\gamma_0\in \dot{C}^\infty(\R;\R^n)$ with $\inf_{\R}|\partial_x\gamma_0|>0$.
    Suppose that $\gamma_0$ is graphical and
    \begin{equation}
        E[\gamma_0]\leq E[\gamma_S]=8-4\sqrt{2}.
    \end{equation}
    Then the elastic flow \eqref{eq:EF} starting from $\gamma_0$ remains graphical for all $t\geq0$.
    The energy threshold is optimal; for any $\varepsilon>0$ there exists a graphical initial curve $\gamma_0$ with $E[\gamma_0]<8-4\sqrt{2}+\varepsilon$ such that the flow is not graphical at some time $t_0>0$.
\end{theorem}

We emphasize that the graphicality-preservation result holds regardless of the codimension.
We also recall that in this case, since $E[\gamma_0]<8$, the convergence to a line is already ensured \cite[Corollary 2.10]{miura2025newenergymethodshortening}.

The proofs of \Cref{thm:EF_embedded} and \Cref{thm:EF_graphical} are essentially variational and the optimal shapes determining the thresholds arise as minimizers of suitable minimization problems for $E$. A challenge in explicitly finding these minimizers is that the constraints involved cause a loss of regularity (non-smoothness) of the optimal shapes, similar to \cite{Miura_Muller_Rupp_2025_optimal}.

We first prove \Cref{thm:EF_graphical} by minimizing $E$ among semi-infinite arcs with a given initial angle, see \Cref{thm:initialangle_borderline_minimality}, yielding part of the \emph{borderline elastica}, see \Cref{sec:borderline}. The proof here is inspired by \cite[Theorem 3.7]{miura2024uniqueness}. As a consequence, we conclude in \Cref{cor:graphical} that the elastic serpent realizes the unique minimizer of $E$ among complete non-graphical curves, proving the graphicality-preserving part of \Cref{thm:EF_graphical}. For the optimality in \Cref{thm:EF_graphical}, some new arguments are required as opposed to \cite{Miura_Muller_Rupp_2025_optimal}.
Inspired by an argument for the loss of graphicality for surface diffusion flows \cite{ElliottMaier-Paape2001}, we carefully construct and analyze a local perturbation of the elastic serpent and show that it leads to a loss of graphicality in finite time. Here the local well-posedness theory for the complete elastic flow in \cite{miura2025newenergymethodshortening} is crucially used.

For \Cref{thm:EF_embedded}, we identify the elastic pendant as the minimizer of $E$ among complete curves with zero rotation number having a tangential self-intersection, see \Cref{thm:elastic_pendant_minimality}. The constraints are motivated by the fact that \eqref{eq:EF} preserves the rotation number and, for stability reasons, the first self-intersection that occurs when starting from an embedded initial datum must be tangential. An important step of independent interest is the variational characterization of the \emph{teardrop elastica} introduced in \cite{Miura_Muller_Rupp_2025_optimal}, see \Cref{thm:teardrop_minimality}.
We also establish a key energy-comparison lemma, see \Cref{lem:06-08}, by using a geometric cut-and-paste argument.
The optimality of the threshold in \Cref{thm:EF_embedded} is then established as in \cite{Miura_Muller_Rupp_2025_optimal} (again using the local well-posedness \cite{miura2025newenergymethodshortening}), whereas the convergence follows from the asymptotic analysis in \cite{miura2025newenergymethodshortening}.

In addition to the above applications to the elastic flow, in \Cref{sec:Li-Yau}, we provide the minimal energy of self-intersecting complete planar curves without further conditions on the rotation number or the nature of the self-intersections, resulting in a \emph{Li--Yau type inequality} in the spirit of \cite{MR4565935}, see \Cref{thm:borderline_minimality}. Along the way, we precisely investigate the energy lower bound of self-intersecting curves in terms of the total curvature. We also discuss difficulties in extending this to higher codimension to obtain an analog to \cite{MR4631455}.

\begin{acknowledgements}
    The first author is supported by JSPS KAKENHI Grant Numbers JP23H00085, JP23K20802, and JP24K00532.
    The second author is funded in whole, or in part, by the Austrian Science Fund (FWF), grant number \href{https://doi.org/10.55776/ESP557}{10.55776/\linebreak ESP557}. Part of this work was done when the second author was visiting Kyoto University supported by a Mobility Fellowship of the Strategic Partnership Program between the University of Vienna and Kyoto University.
\end{acknowledgements}

\section{Preliminaries}

Throughout this paper we will use the fact that the Sobolev space $W^{m,p}(I)$ is embedded in $C^{m-1}(\bar{I})$ for any $m\in\N$, $p\geq1$, and any interval $I\subset\R$.

\subsection{Adapted elastic energy}
We first note the following relation between the energies for curves closing up continuously.
\begin{lemma}\label{lem:energy_C0-closed}
    Let $I=(a,b)\subset \R$ be a bounded interval and $\gamma\in W^{1,1}(I;\R^n)$ be such that $\operatorname{ess\,inf}_I|\partial_x\gamma|>0$ and $\gamma(a)=\gamma(b)$. Then $D[\gamma]=L[\gamma]$ and, if $\gamma\in W^{2,2}(I;\R^n)$, then $E[\gamma]=\hat E[\gamma]$.
\end{lemma}
\begin{proof}
    This follows from integration by parts, noting that
    \begin{align}
        D[\gamma] = \int_I(1-\langle \partial_s\gamma, e_1\rangle) ds = L[\gamma] - \langle \gamma, e_1\rangle \vert_a^b. &\qedhere
    \end{align}
\end{proof}

We next prove the horizontality of the ends under the finiteness of the adapted elastic energy $E$, which was previously shown under the finiteness of $D$ and $\dot{C}^{1,1}$ regularity, see \cite[Section 3]{miura2025newenergymethodshortening}.

\begin{proposition}\label{prop:finite_E_horizontal_end}
    Let $\gamma\in W^{2,2}_{loc}(\R;\R^{n})$ with $\inf_\R|\partial_x\gamma|>0$ and $E[\gamma]<\infty$.
    Then $\lim_{x\to\pm\infty}\partial_s\gamma(x)=e_1$.
\end{proposition}

We begin with the following auxiliary lemma.

\begin{lemma}\label{lem:finite_E_tangent_BV}
    Let $I\subset\R$ be an interval and $\gamma\in W^{2,2}_{loc}(I;\R^{n})$ with $\inf_I
    |\partial_x\gamma|>0$ and $E[\gamma]<\infty$.
    Then the function $\langle\partial_s\gamma,e_1\rangle:I\to[-1,1]$ is of bounded variation.
\end{lemma}

\begin{proof}
    For any $\{x_j\}_{j=1}^M\subset I$ with $x_1<x_2<\dots<x_M$, we have
    \begin{align}
        E[\gamma]&\geq %\sum_{j=1}^{M-1} E[\gamma|_{[x_j,x_{j+1}]}] \\
         \sum_{j=1}^{M-1}\int_{x_j}^{x_{j+1}}\left(\frac{1}{2}|\partial_s^2\gamma|^2+\frac{1}{2}|\partial_s\gamma - e_1|^2\right) ds \\
        &\geq \sum_{j=1}^{M-1}\int_{x_j}^{x_{j+1}}|\partial_s^2\gamma||\partial_s\gamma - e_1| ds \\
        &\geq \sum_{j=1}^{M-1}\left| \int_{x_j}^{x_{j+1}}\partial_s\left(\frac{1}{2}|
        \partial_s\gamma - e_1|^2\right) ds \right| \\
        %&= \sum_{j=1}^{M-1}\frac{1}{2} \left| |\partial_s\gamma(x_{j+1}) - e_1|^2 - |\partial_s\gamma(x_j) - e_1|^2\right|\\
        &= \sum_{j=1}^{M-1} |\langle\partial_s\gamma(x_{j+1}),e_1\rangle-\langle\partial_s\gamma(x_j),e_1\rangle|.
    \end{align}
    The arbitrariness of the partition $\{x_j\}$ completes the proof.
\end{proof}

\begin{proof}[Proof of \Cref{prop:finite_E_horizontal_end}]
    By \Cref{lem:finite_E_tangent_BV} we deduce that $\langle\partial_s\gamma,e_1\rangle$ converges as $x\to\pm\infty$.
    Since $D[\gamma]\leq E[\gamma]<\infty$, the limits must be $1$, i.e., $\partial_s\gamma\to e_1$ as $x\to\pm\infty$.
\end{proof}

\subsection{Rotation number}\label{subsec:rotation_number}

For a planar immersed curve $\gamma:I\to\R^2$, defined on an interval $I\subset\R$, possibly unbounded, we define the rotation number by
\[
N[\gamma]\vcentcolon= \frac{1}{2\pi}\int_I k ds,
\]
where $k$ denotes the signed curvature (such that a counterclockwise circle has positive curvature) and the integral is understood as an improper integral.
For a tangential angle function $\theta\colon I\to\R$ such that $\partial_s\gamma=(\cos \theta, \sin\theta)$, we have $k=\partial_s\theta$, so thanks to \Cref{prop:finite_E_horizontal_end}, the rotation number is well defined for all curves with finite adapted elastic energy.

\subsection{Borderline elastica}\label{sec:borderline}

We recall basic facts on the (planar) borderline elastica (see also \cite{Miura_elastica_survey}).
A prototypical arclength parametrization is given by
\begin{equation}\label{eq:borderline}
    \gamma_b(s)\vcentcolon=
    \begin{pmatrix}
        s-2\tanh{s}\\
        2\sech{s}
    \end{pmatrix}.
\end{equation}
The tangential angle of $\gamma_b$ is given by
$\theta_b(s)=4\arctan(e^{s})$, which increases from $0$ to $2\pi$ as $s$ varies from $-\infty$ to $\infty$.
The signed curvature is then given by
$k_b(s)=\partial_s\theta_b(s)=2\sech{s}=4e^{s}(1+e^{2s})^{-1}$. Note that $k_b\geq 0$.
This satisfies $k_{ss}+\frac{1}{2}k^3-k=0$.
In particular, we can explicitly compute:
\begin{align}
    E[\gamma_b] &=\int_\R \left( \frac{1}{2}k_b^2 + 1-\cos\theta_b \right) ds = \int_\R \left( \frac{1}{2} k_b^2 + \frac{8\tan^2\frac{\theta_b}{4}}{(1+\tan^2\frac{\theta_b}{4})^2} \right) ds \\
    &= \int_\R \frac{16e^{2s}}{(1+e^{2s})^2} ds = \left[ -\frac{8}{1+e^{2s}} \right]_{-\infty}^\infty = 8. \label{eq:energy_8_borderline}
\end{align}

\begin{figure}[htbp]
    \centering
    \includegraphics[width=0.6\linewidth]{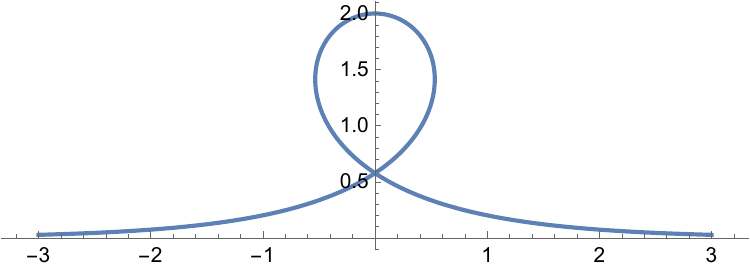}
    \caption{The borderline elastica, parametrized as in \eqref{eq:borderline}.}
    \label{fig:borderline}
\end{figure}

\section{Graphicality preservation}

The goal of this section is to prove the graphicality-preservation property of the elastic flow, \Cref{thm:EF_graphical}.
An important observation here is that a piece of the borderline elastica can be used to produce an optimal threshold for the energy $E$ below which all curves are graphical.

We define the borderline elastica with initial angle (cf.\ \cite{Miura20}) using the notation $\gamma_b,\theta_b,k_b$ introduced above. 
Given $\varphi\in(0,\pi]$, let $s_\varphi\in\R$ be the unique number such that $\theta_b(s_\varphi)+\varphi=2\pi$.
Clearly $s_\varphi\geq0$, and it is precisely given by
\begin{align}
    s_\varphi &= \log(e^{s_\varphi}) = \log\tan\frac{\theta_b(s_\varphi)}{4} = \log\tan\left(\frac{\pi}{2}-\frac{\varphi}{4}\right) = \log\cot\frac{\varphi}{4}.
\end{align}
Let $\gamma_b^\varphi:[0,\infty)\to\R^2$ denote the borderline elastica associated with $\varphi$ defined by
\[
\gamma_b^\varphi\vcentcolon=\gamma_b(s+s_\varphi)-\gamma_b(s_\varphi).
\]
This satisfies $\gamma_b^\varphi(0)=(0,0)$, $\partial_s\gamma_b^\varphi(0)=(\cos\varphi,-\sin\varphi)$, and $\partial_s\gamma_b^\varphi(s)\to e_1$ as $s\to\infty$, see \Cref{fig:borderline_angle}.
Also, computing as in \eqref{eq:energy_8_borderline} yields
\begin{equation}\label{eq:borderline_angle_energy}
    E[\gamma_b^\varphi] = \int_{s_\varphi}^\infty \frac{16e^{2s}}{(1+e^{2s})^2} ds = \frac{8}{1+e^{2s_\varphi}} = \frac{8}{1+\cot^2\frac{\varphi}{4}}=8\sin^2\frac{\varphi}{4}.
\end{equation}

\begin{figure}[htbp]
    \centering
    \includegraphics[width=0.5\linewidth]{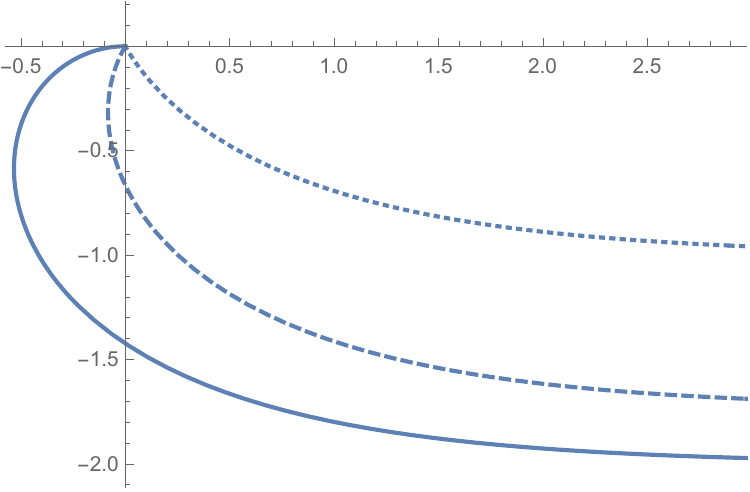}
    \caption{The borderline elastica $\gamma_b^\varphi$ with initial angle $\varphi\in \{\pi, \frac{2}{3}\pi, \frac{1}{3}\pi \}$.}
    \label{fig:borderline_angle}
\end{figure}

Now we obtain the key variational characterization of the above curves.
We call $\Phi:\R^n\to\R^n$ a \emph{direction-preserving isometry} (in the direction $e_1$) if there are a vector $b\in\R^n$ and an orthogonal matrix $A\in O(n)$ with $Ae_1=e_1$ such that $\Phi(x)=Ax+b$ for $x\in\R^n$.

\begin{theorem}\label{thm:initialangle_borderline_minimality}
    Let $\varphi\in(0,\pi]$.
    Let $\gamma:[0,\infty)\to\R^n$ be an immersion such that $\inf_\R|\partial_x\gamma|>0$ and $\gamma|_{[0,R]}\in W^{2,2}(0,R;\R^n)$ for any $R>0$.
    Suppose that 
    \[\langle\partial_s\gamma(0),e_1\rangle\leq\cos\varphi.\] 
    Then
    \[
    E[\gamma] \geq 8\sin^2\frac{\varphi}{4},
    \]
    where equality is attained if and only if $\gamma$ coincides with $\gamma_b^\varphi$ up to reparametrization and direction-preserving isometry.
\end{theorem}

The proof is almost parallel to \cite[Theorem 3.7]{miura2024uniqueness}, which essentially corresponds to the case of $\varphi=\pi$ and $n=2$.
For the reader's convenience we sketch the argument.
    
\begin{proof}[Proof of \Cref{thm:initialangle_borderline_minimality}]
    Without loss of generality we may assume that all the curves under consideration are parametrized by arclength.
    Let $A_\varphi$ be the class of arclength parametrized curves $\gamma:[0,\infty)\to\R^n$ such that $\gamma|_{[0,R]}\in W^{2,2}(0,R;\R^n)$ for any $R>0$, and $\langle\partial_s\gamma(0),e_1\rangle\leq\cos\varphi$.
    
    We first show that $E$ has a minimizer in $A_\varphi$.
    Take a minimizing sequence $\{\gamma_j\}\subset A_\varphi$ for $E$.
    Up to translation we may assume $\gamma_j(0)=0$ for all $j$.
    Then $\{\gamma_j|_{[0,R]}\}_j$ is bounded in $W^{2,2}(0,R;\R^n)$ for any $R>0$.
    Hence, by a compactness and a diagonal argument, there are a subsequence of $\{\gamma_j\}$ (without relabeling) and some $\gamma_\infty:[0,\infty)\to\R^n$ such that $\gamma_\infty|_{[0,R]}\in W^{2,2}(0,R;\R^n)$ for any $R>0$, with the property that for any $R>0$ the restriction $\gamma_j|_{[0,R]}$ converges to $\gamma_\infty|_{[0,R]}$ in the $W^{2,2}$-weak and $C^1$-strong topology.
    Then $\gamma_\infty\in A_\varphi$ and, by lower semicontinuity, $\liminf_{j\to\infty}E[\gamma_j]\geq E[\gamma_\infty]$.
    Hence the limit curve $\gamma_\infty$ is a minimizer.

    By minimality, $\gamma_\infty$ must be an elastica (since the integrand of $L$ and of $D$ agree up to a null Lagrangian, see \cite{miura2024uniqueness} for details).
    By using the classification of elasticae (see, e.g., \cite{Singer_lecturenotes,Miura_elastica_survey}), the finiteness of $E$ implies that $\gamma_\infty$ is either a straight semi-line or a part of an isometric image of the borderline elastica $\gamma_b$.
    By $\gamma_\infty\in A_\varphi$ and by finiteness of the direction energy,
    \begin{align}
        \langle\partial_s\gamma_\infty(0),e_1\rangle\in[-1,\cos\varphi]\subset[-1,1), \qquad \lim_{s\to\infty}\partial_s\gamma_\infty(s)=e_1.
    \end{align}
    Hence clearly, no straight line is admissible.
    In addition, by symmetry of $\gamma_b$, any infinite-length piece of $\gamma_b$ containing the vertex of the loop has more energy than $\gamma_b^\pi$.
    Hence we find that, up to direction-preserving isometry, the curve $\gamma_\infty$ coincides with $\gamma_b^{\varphi'}$ for some $\varphi'\in[\varphi,\pi]$.
    Now from \eqref{eq:borderline_angle_energy} we infer that the case $\varphi'=\varphi$ is the only possibility where the desired minimal energy $8\sin^2\frac{\varphi}{4}$ is attained.
    The proof is complete.
\end{proof}

\begin{corollary}\label{cor:serpent_general}
    Let $\gamma\in W^{2,2}_{loc}(\R;\R^{n})$ with $\inf_\R|\partial_x\gamma|>0$.
    If there are some $x_0\in\R$ and $\varphi_0\in(0,\pi]$ such that $\langle \partial_s\gamma(x_0),e_1\rangle \leq \cos\varphi_0$, then
    \begin{equation}
        E[\gamma] \geq 16\sin^2\frac{\varphi_0}{4},
    \end{equation}
    where equality holds if and only if both $\gamma_1,\gamma_2:[0,\infty)\to\R^n$ defined by
    \begin{equation}
        \gamma_1(x)\vcentcolon=\gamma(x-x_0)|_{[x_0,\infty)}, \qquad \gamma_2(x)\vcentcolon=-\gamma(x_0-x)|_{[(-\infty,x_0]},
    \end{equation}
    coincide with $\gamma_b^{\varphi_0}$ up to reparametrization and direction-preserving isometry.
\end{corollary}

\begin{proof}
    It directly follows by observing that $E[\gamma]=E[\gamma_1]+E[\gamma_2]$ and applying \Cref{thm:initialangle_borderline_minimality} to both $\gamma_1$ and $\gamma_2$.
\end{proof}

The case $\varphi_0=\pi/2$ is particularly important in view of graphicality.

\begin{definition}\label{def:elastic_serpent}
    We define the elastic serpent $\gamma_S\in \dot{C}^{1,1}(\R;\R^2)$ by concatenating $\gamma_b^{\pi/2}$ with its reflection at the origin, i.e.,
    \begin{align}\label{eq:elastic_serpent}
        \gamma_S(s) \vcentcolon= \begin{cases}
        \gamma_b^{\pi/2}(s), &s\geq 0, \\
        -\gamma_b^{\pi/2}(-s), & s\leq 0,
    \end{cases}
    \end{align}
    see \Cref{fig:elastic_serpent}.
\end{definition}

The elastic serpent is parametrized by arclength, and nothing but the unique curve (up to invariances) attaining the equality case in \Cref{cor:serpent_general} with $\varphi_0=\frac{\pi}{2}$.
In particular,
\begin{equation}
    E[\gamma_S]=16\sin^2\frac{\pi}{8}=8-4\sqrt{2}.
\end{equation}
This energy provides an optimal threshold for graphicality.

\begin{corollary}\label{cor:graphical}
    Let $\gamma\in W^{2,2}_{loc}(\R;\R^{n})$ with $\inf_\R|\partial_x\gamma|>0$.
    If \[E[\gamma]<8-4\sqrt{2},\] then $\gamma$ is graphical.
\end{corollary}

\begin{proof}
    We prove the contrapositive.
    If $\gamma$ is not graphical, then there are sequences $\{x_j\}\subset\R$ and $\varphi_j\to\frac{\pi}{2}$ such that $\langle \partial_s\gamma(x_j),e_1\rangle \leq\cos\varphi_j$.
    Then by \Cref{cor:serpent_general} we deduce that $E[\gamma]\geq 16\sin^2\frac{\varphi_j}{4}\to 16\sin^2\frac{\pi}{8}=8-4\sqrt{2}$.
\end{proof}

The above corollary almost directly implies the graphicality-preserving part of \Cref{thm:EF_graphical}.
To prove its optimality, inspired by the approach in \cite{ElliottMaier-Paape2001}, we construct a perturbation of the elastic serpent $\gamma_S$ which loses graphicality in finite time under the elastic flow.

\begin{lemma}\label{lem:optimality_eta_alpha}
    Let $\varepsilon>0$. Then there exists a family of properly immersed smooth curves $\{\eta_\alpha\}_{\alpha\in[0,1]}\subset\dot{C}^\infty(\R;\R^2)$ such that
    \begin{enumerate}
        \item\label{item:opti_energy} $E[\eta_\alpha]\leq E[\gamma_S]+\varepsilon$ for all $\alpha\in[0,1]$;
        \item\label{item:opti_graph} $\inf_\R \langle \partial_s\eta_\alpha,e_1\rangle>0$ for all $\alpha\in(0,1]$;
        \item\label{item:opti_convergence} $\eta_\alpha\to\eta_0$ smoothly as $\alpha \searrow 0$;
        \item \label{item:opti_tangent_eta0} $\langle \partial_s\eta_0(0),e_1\rangle =0$;
        \item\label{item:opti_velocity} 
        $\langle (\nabla_s(\nabla_s^2\kappa+\frac{1}{2}|\kappa|^2\kappa-\kappa)-\xi\kappa)[\eta_0]\vert_{x=0},e_1\rangle>0$ for all smooth $\xi\colon\R\to\R$.
    \end{enumerate}
\end{lemma}
\begin{proof}
As $\partial_s\gamma_S(0)=(0,-1)$, we may locally near the origin write $\gamma_S$ as a graph over the $e_2$-axis of a $C^1$-function $v\colon(-\rho_0,\rho_0)\to\R$, $0<\rho_0\leq 1$, which is odd, smooth away from the origin, and satisfies $v(0)=0, v'(0)=0$. In fact, we have $v\in W^{2,\infty}((-\rho_0,\rho_0))$ and 
\begin{align}\label{eq:v_bounds}
    |v(x)|\leq C x^2 \quad \text{ and } \quad 0< v'(x)\leq C |x| \quad \text{ for all }x\neq 0.
\end{align}
Moreover, by direct computation, we have
\begin{align}\label{eq:26-08}
    \partial_s^2 \gamma_b^{\pi/2}(0) = (\sqrt{2},0)
\end{align}
so that, possibly reducing $\rho_0>0$, there exists $\beta>0$ with
\begin{align}\label{eq:v_quadratic}
    v(x)\geq \beta x^2 \quad\text{ for all }x\geq 0.
\end{align}
Let $u_\alpha(x)\vcentcolon = x^5+\alpha x$ and let $\psi\in C_c^\infty(\R)$ be an even cutoff function with $\chi_{[-1/2,1/2]}\leq \psi\leq \chi_{(-1,1)}$ and $-\psi'(-x)=\psi'(x)\leq 0$ for $x\geq 0$. For $0\leq \rho< \min\{\rho_0,\beta/4\}$, we replace $v$ by the smooth function $w_\alpha$, $\alpha\in [0,1]$, given by
\begin{align}
    w_\alpha(x) \vcentcolon = \big(1-\psi(x/\rho)\big)v(x) + \rho^2\psi(x/\rho)u_\alpha(x),
\end{align}
so that $w_\alpha = v$ in $(-\rho_0,-\rho]\cup[\rho,\rho_0)$ and $w_\alpha = \rho^2 u_\alpha$ in $[-\rho/2,\rho/2]$. Similarly as in \cite[(3.4)]{Miura_Muller_Rupp_2025_optimal}, one can use \eqref{eq:v_bounds} to show that 
\begin{align}\label{eq:opti_H2_control}
    \Vert v-w_\alpha\Vert_{H^2(-\rho_0,\rho_0)}^2 \leq C\rho.
\end{align}

We now show that $w_\alpha'>0$ for $\alpha>0$. Given the properties of $v$ and $u_\alpha$ and using symmetry, it suffices to consider $x\in [\rho/2,\rho]$. Since $\rho\leq \rho_0\leq 1$, we have
\begin{align}
    \rho^2u_\alpha(x) = \rho^2(\alpha x + x^5) \leq 2\rho^2 x \leq 4\rho x^2 \leq \beta x^2\leq v(x),
\end{align}
by \eqref{eq:v_quadratic} using $\rho\leq \beta/4$. Moreover, $\psi'\leq 0$ on $[0,\infty)$, yields that 
\begin{align}
    w_\alpha'(x) = \big(1-\psi(x/\rho)\big)v'(x) + \rho^2\psi(x/\rho)u_\alpha'(x) + \big(\rho^2u_\alpha(x)-v(x)\big)\rho^{-1}\psi'(x/\rho) >0.
\end{align}
For $\rho>0$ sufficiently small and $\alpha\in [0,1]$, let $\eta_\alpha$ be the curve obtained by locally replacing $v$ by $w_\alpha$. Clearly \eqref{item:opti_convergence} and \eqref{item:opti_tangent_eta0} are satisfied and \eqref{item:opti_graph} follows from $w_\alpha'>0$. Since $E$ is continuous with respect to compactly supported $H^2$-perturbations, \eqref{item:opti_energy} is a consequence of \eqref{eq:opti_H2_control} for $\rho=\rho(\varepsilon)>0$ small enough, independent of $\alpha\in[0,1]$. Since $\partial_x w_0(0)=0$ for $x=1,\dots, 4$, by explicit computations in local coordinates, see, for instance, \cite[(A4)]{MR4278396}, we conclude $\nabla_s^k\kappa[\eta_0]\vert_{x=0} = 0$ for $k=0,1,2$, and thus 
\begin{align}
    \Big\langle \big(\nabla_s(\nabla_s^2\kappa +\frac{1}{2}|\kappa|^2\kappa -\kappa)-\xi\kappa\big)[\eta_0]\vert_{x=0},e_1\Big\rangle =  \partial_x^5 w_0(0)=   5!  \rho^2.&\qedhere
\end{align}
\end{proof}

Together with local well-posedness arguments, see \cite[Appendix A]{miura2025newenergymethodshortening}, we can now conclude the optimality in \Cref{thm:EF_graphical}.

\begin{lemma}\label{lem:eta_alpha_breaks}
    Let $\varepsilon>0$ and let $\eta_\alpha$ be as in \Cref{lem:optimality_eta_alpha}.
    For $\alpha>0$ small enough, the unique elastic flow $\gamma_\alpha$ with initial datum $\eta_\alpha$ becomes non-graphical in finite time, i.e., there exists $t\in(0,\infty)$ such that $\inf_\R \langle \partial_s\gamma_\alpha(t,\cdot),e_1\rangle <0$.
\end{lemma}
\begin{proof}
    Recall from \cite[Appendix A]{miura2025newenergymethodshortening} that the unique elastic flow with initial datum $\gamma_0$ is given as a time-dependent reparametrization $\gamma(t,\Phi_t)$ of the unique solution $\gamma$ to the so-called analytic problem 
    \begin{align}\label{eq:STE_IVP_analytic}
        \left\lbrace\begin{array}{ll}
             \partial_t \gamma &= -\nabla_s^2\kappa-\frac{1}{2}|\kappa|^2\kappa+\kappa + \xi[\gamma] \partial_s\gamma\\
             \gamma(0) &= \gamma_0,
        \end{array}\right.
    \end{align}
    where the tangential velocity $\xi[\gamma]$ makes the equation parabolic. Let $T>0$ be as in \cite[Theorem A.2]{miura2025newenergymethodshortening} for \eqref{eq:STE_IVP_analytic} with initial datum $\eta_0$. 
    By \cite[Theorem A.2]{miura2025newenergymethodshortening} and \Cref{lem:optimality_eta_alpha}~\eqref{item:opti_convergence}, the solution $\Gamma_\alpha$ to the initial value problem \eqref{eq:STE_IVP_analytic} corresponding to the elastic flow exists on $[0,T]$ for all $\alpha\in[0,\alpha_0]$ with sufficiently small $\alpha_0>0$. By \cite[Theorem A.3]{miura2025newenergymethodshortening}, there exists a time-dependent family of orientation-preserving diffeomorphisms $\Phi_t^\alpha\colon\R\to\R$ such that $\gamma_\alpha(t,x) = \Gamma_\alpha(t,\Phi_t^\alpha(x))$ for $t\in [0,T]$, $x\in\R$, and $\alpha\in[0,\alpha_0]$. By  \cite[Theorem A.2]{miura2025newenergymethodshortening}, the map $F\colon [0,\alpha_0]\times[0,T]\to\R$ defined by
    \begin{align}
        F(\alpha,t)\vcentcolon= \langle \partial_s \Gamma_\alpha(t,0),e_1\rangle
    \end{align}
    is continuous, where $\partial_s$ is with respect to $\Gamma_\alpha$. By \Cref{lem:optimality_eta_alpha}~\eqref{item:opti_tangent_eta0}, we have $F(0,0)=0$. Moreover, with $\xi\vcentcolon=\langle \partial_t\Gamma_0,\partial_s\Gamma_0\rangle$, equation \eqref{eq:STE_IVP_analytic}, standard geometric evolution formulae as in \cite[(2.5)]{Dziuk-Kuwert-Schatzle_2002}, and \Cref{lem:optimality_eta_alpha}~\eqref{item:opti_velocity} yield
    \begin{align}
        \partial_t F(0,0)&= \langle \partial_t \partial_s\Gamma_0(t,x),e_1\rangle\vert_{t=x=0} \\
        &= - \Big\langle \big(\nabla_s(\nabla_s^2\kappa +\frac{1}{2}|\kappa|^2\kappa -\kappa)-\xi\kappa\big)[\eta_0]\vert_{x=0},e_1\Big\rangle <0.
    \end{align}
    By continuity, there is $t\in (0,T)$ such that for all $\alpha>0$ small enough (depending on $t$), we have
    $F(\alpha,t)<0$. Now, since $\Phi_t^\alpha$ is an orientation-preserving diffeomorphism,
    \begin{align}
        \inf_\R\langle \partial_s\gamma_\alpha(t,\cdot),e_1\rangle = \inf_\R\langle \partial_s\Gamma_\alpha(t,\cdot),e_1\rangle <0. &\qedhere
    \end{align}
\end{proof}

\begin{proof}[Proof of \Cref{thm:EF_graphical}]
    By the energy decay \eqref{eq:energy_decay}, either the flow $\gamma$ is stationary or the energy satisfies $E[\gamma(t,\cdot)]<8-4\sqrt{2}$ for all $t>0$.
    The former case clearly preserves the graphicality (in fact the only possibility is a line) while in the latter case we deduce the desired graphicality-preservation property by \Cref{cor:graphical}.

    For the optimality of the energy threshold, for any $\varepsilon>0$, \Cref{lem:optimality_eta_alpha,lem:eta_alpha_breaks} yield a curve $\eta_\alpha$ with $E[\eta_\alpha]\leq 8-4\sqrt{2}+\varepsilon$ which is graphical, but the resulting elastic flow loses graphicality in finite time.
\end{proof}

\begin{remark}
    \Cref{cor:serpent_general} shows that the smallness of $E$ implies not only graphicality but also the global gradient smallness.
    However, zeroth order quantities cannot globally be controlled under any smallness of $E$.
    In fact, if we take $\gamma_j$ as the graph of
    $
    u_j(x)\vcentcolon=(x^2+j)^{\alpha}\sin\log(x^2+j),
    $
    where $\alpha\in(0,\frac{1}{4})$, then
    $\lim_{j\to\infty}E[\gamma_j]=0$, see \cite[Lemma 3.2]{miura2025newenergymethodshortening},
    while each $u_j$ diverges with oscillations as $x\to\pm\infty$.
    In particular, the smallness assumption in \Cref{thm:EF_graphical} admits initial data of graphs of unbounded functions.
\end{remark}

\section{Planar embeddedness preservation}

In this section we prove the planar embeddedness-preservation property, \Cref{thm:EF_embedded}.
Here, not only a piece of the borderline elastica, but also the teardrop elastica introduced in \cite[Definition 2.15]{Miura_Muller_Rupp_2025_optimal} is importantly used to produce an optimal threshold.

We first give a key variational characterization of the teardrop elastica, using the minimality of the elastic two-teardrop obtained in \cite[Theorem 1.4]{Miura_Muller_Rupp_2025_optimal}.

To this end we briefly recall a prototypical parametrization of the teardrop elastica and of the elastic two-teardrop.
Here we use the same notation as in \cite[Appendix A]{Miura_elastica_survey}; in particular, we define the incomplete elliptic integral of first kind $F(x,m)$ and of second kind $E(x,m)$ by
\[
F(x,m)\vcentcolon=\int_0^x\frac{d\theta}{\sqrt{1-m\sin^2\theta}}, \qquad E(x,m)\vcentcolon=\int_0^x\sqrt{1-m\sin^2\theta}d\theta.
\]
Let $\gamma_w(\cdot,m):\R\to\R^2$ be a one-parameter family of wavelike elasticae, where $m\in(0,1)$, defined by
\begin{align}
    \gamma_w(x,m)\vcentcolon=
    \begin{pmatrix}
        2E(x,m)-F(x,m)\\
        -2\sqrt{m}\cos{x}
    \end{pmatrix}.
\end{align}
The signed curvature of $\gamma_w(\cdot,m)$ at $x$ is given by $2\sqrt{m}\cos{x}$.
By \cite[Proposition 2.13]{Miura_Muller_Rupp_2025_optimal} we can define $m_T\in(\frac{1}{2},1)$ as the unique root of the function
\begin{equation}
    m \mapsto \int_0^{\pi-\alpha(m)}\frac{1-2m\sin^2{\theta}}{\sqrt{1-m\sin^2\theta}}d\theta, \qquad \alpha(m)\vcentcolon=\arcsin{\sqrt{\frac{1}{2m}
    }}\in(0,\pi/2).
\end{equation}
Numerically, $m_T\approx0.731183$.
We then define a (non-arclength) parametrization of the teardrop elastica by
\begin{equation}\label{eq:teardrop_before_rescaling}
    \gamma_{T}\vcentcolon= \gamma_w(\cdot,m_T)|_{[-\pi+\alpha(m_T),\pi-\alpha(m_T)]}.
\end{equation}
Since the endpoints of $\gamma_{T}$ agree and the tangent directions are opposite there, we can define the elastic two-teardrop $\gamma_{2T}$ by concatenating $\gamma_{T}$ and its point reflection across the endpoints.
We finally recall the precise statement about the minimality of $\gamma_{2T}$ in terms of the scale-invariant energy $LB$.

\begin{theorem}[{\cite[Theorem 1.4]{Miura_Muller_Rupp_2025_optimal}}]\label{thm:two-teardrop}
    Let $\gamma\in W^{2,2}(\mathbf{S}^1;\R^2)$ be an immersed closed planar curve with unit (absolute) rotation number $|N[\gamma]|=1$.
    Then 
    \begin{equation}
        L[\gamma]B[\gamma] \geq L[\gamma_{2T}]B[\gamma_{2T}],
    \end{equation}
    where equality holds if and only if $\gamma$ coincides with $\gamma_{2T}$ up to reparametrization and similarity.
\end{theorem}

Now we can formulate and prove the key variational property of the teardrop elastica.

\begin{theorem}\label{thm:teardrop_minimality_LB}
    Let $\gamma\in W^{2,2}(0,1;\R^2)$ be an immersion such that $\gamma(0)=\gamma(1)$ and $\partial_s\gamma(0)=-\partial_s\gamma(1)$.
    Then 
    \begin{equation}
        L[\gamma]B[\gamma] \geq L[\gamma_{T}]B[\gamma_{T}],
    \end{equation}
    where equality holds if and only if $\gamma$ coincides with $\gamma_{T}$ in \eqref{eq:teardrop_before_rescaling} up to reparametrization and similarity.
\end{theorem}

\begin{proof}
    It is sufficient to prove existence and uniqueness (up to invariances) of minimizers of $LB$ among admissible curves.
    Without loss of generality, we may take a minimizing sequence within the class of curves $\gamma$ that are parametrized by arclength, have length $1$, and satisfy $\gamma(0)=\gamma(1)=0$ and $\partial_s\gamma(0)=-\partial_s\gamma(1)=e_1$.
    Then, existence follows by a standard direct method and any minimizer $\bar{\gamma}$ is a smooth elastica, cf.\ \cite{Miura_elastica_survey}.
    It remains to show that $\bar{\gamma}$ coincides with $\gamma_{T}$ up to reparametrization and similarity.
    
    We prove that $|N[\bar{\gamma}]|=\frac{1}{2}$.
    By the boundary condition it is clear that $N[\bar{\gamma}]\in \Z+\frac{1}{2}$.
    Suppose on the contrary that $|N[\bar{\gamma}]|\geq\frac{3}{2}>1$. 
    Then by the well-known classification of planar elasticae (see \cite{Singer_lecturenotes,Miura_elastica_survey}) the only possibility is an orbitlike elastica with more than one period of the curvature, but this contradicts the minimality \cite[Theorem 2.1]{MR4739246}.
    
    Then we can create a closed $W^{2,2}$-curve $\hat{\gamma}$ such that $\hat{\gamma}$ has a self-intersection and $|N[\hat{\gamma}]|=1$, by concatenating $\bar{\gamma}$ and its point reflection across the endpoint.
    Then by \Cref{thm:two-teardrop} we have $L[\hat{\gamma}]B[\hat{\gamma}] \geq L[\gamma_{2T}]B[\gamma_{2T}]$ and, in fact, equality holds; indeed, if the inequality is strict, by symmetry it means that $L[\bar{\gamma}]B[\bar{\gamma}]>L[\gamma_{T}]B[\gamma_{T}]$, but this contradicts the minimality of $\bar{\gamma}$.
    Therefore, by uniqueness in the equality case of \Cref{thm:two-teardrop}, we deduce that $\hat{\gamma}$ coincides with $\gamma_{2T}$ up to reparametrization and similarity.
    This implies the desired uniqueness of $\bar{\gamma}$.
\end{proof}

For our purpose it is more convenient to reformulate the above result in terms of the additive energy $\hat{E}[\gamma]=B[\gamma]+L[\gamma]$.

Recall that, for $\lambda\in\R$, a curve $\gamma$ is called a \emph{$\lambda$-elastica} (in \cite{Miura_elastica_survey} for example) if it is a critical point of the energy $\int_\gamma(|\kappa|^2+\lambda)ds$, which is $2B+\lambda L$ in the current notation, yielding the Euler--Lagrange the equation $2\nabla_s^2\kappa+|\kappa|^2\kappa-\lambda\kappa=0$.
For a $\lambda$-elastica $\gamma_\lambda$ and a scaling factor $\Lambda>0$, we have $\hat{E}[\Lambda\gamma_\lambda]=\frac{1}{2\Lambda}\int_{\gamma_\lambda}(|\kappa|^2+2\Lambda^2)ds$.
Therefore, $\Lambda\gamma_\lambda$ is a critical point of $\hat{E}$ if and only if $\lambda>0$ and $\Lambda=\sqrt{\lambda/2}$.

The teardrop elastica $\gamma_{T}$ defined in \eqref{eq:teardrop_before_rescaling} is a $2(2m_T-1)$-elastica \cite[Remark 2.19]{Miura_elastica_survey}.
Hence, the curve $\Lambda\gamma_{T}$ with $\Lambda>0$ is a critical point of $\hat{E}$ if and only if $\Lambda=\sqrt{2m_T-1}$.
By the criticality and a scaling argument we have $B[\sqrt{2m_T-1}\gamma_{T}]=L[\sqrt{2m_T-1}\gamma_{T}]$.
Based on these observations, we now define
\begin{equation}\label{eq:teardrop_after_rescaling}
    \hat{\gamma}_{T}:\big[ 0,L[\hat \gamma]\big]\to\R^2
\end{equation}
by the arclength reparametrization of $\sqrt{2m_T-1}\gamma_{T}$.
Then $\hat{\gamma}_{T}$ is a critical point of $\hat{E}$, and thus satisfies
\begin{equation}\label{eq:B=L}
    B[\hat{\gamma}_{T}]=L[\hat{\gamma}_{T}].
\end{equation}

\begin{theorem}\label{thm:teardrop_minimality}
    Let $\gamma\in W^{2,2}(0,1;\R^2)$ be an immersion such that $\gamma(0)=\gamma(1)$ and $\partial_s\gamma(0)=-\partial_s\gamma(1)$.
    Then 
    \begin{equation}
        \hat{E}[\gamma] \geq \hat{E}[\hat{\gamma}_{T}],
    \end{equation}
    where equality holds if and only if $\gamma$ coincides with $\hat{\gamma}_{T}$ in \eqref{eq:teardrop_after_rescaling} up to reparametrization and isometry.
\end{theorem}

\begin{proof}
    The Cauchy--Schwarz inequality and \Cref{thm:two-teardrop} yield
    \begin{align}\label{eq:0417-01}
        \hat{E}[\gamma]=B[\gamma] + L[\gamma] \geq 2\sqrt{L[\gamma]B[\gamma]} \geq 2\sqrt{L[\gamma_{T}]B[\gamma_{T}]}.
    \end{align}
    By \Cref{thm:two-teardrop}, the equality case in the last inequality occurs if and only if $\gamma$ coincides with $\gamma_{T}$ up to reparametrization and similarity.
    Moreover, in the first inequality, equality holds if and only if $B[\gamma]=L[\gamma]$.
    Noting that, thanks to \eqref{eq:B=L}, $B[\Lambda\hat{\gamma}_{T}]=\Lambda^{-1}B[\hat{\gamma}_{T}]$ and $L[\Lambda\hat{\gamma}_{T}]=\Lambda L[\gamma]$ coincide for $\Lambda>0$ if and only if $\Lambda=1$, we deduce that equality in \eqref{eq:0417-01} holds if and only if $\gamma$ coincides with $\hat{\gamma}_{T}$ up to reparametrization and isometry.
    In addition, the scaling-invariance of $LB$ and \eqref{eq:B=L} yield
    \begin{align}
        2\sqrt{L[\gamma_{T}]B[\gamma_{T}]} = 2\sqrt{L[\hat{\gamma}_{T}]B[\hat{\gamma}_{T}]} = \hat{E}[\hat{\gamma}_{T}].
    \end{align}
    This together with \eqref{eq:0417-01} completes the proof.
\end{proof}

\begin{remark}\label{rem:teardrop_minimality}
    Under the assumption of \Cref{thm:teardrop_minimality}, \Cref{lem:energy_C0-closed} also yields
    \begin{equation}
        E[\gamma]\geq E[\hat{\gamma}_{T}],
    \end{equation}
    and the same rigidity in the equality case.
    Numerically, $E[\hat{\gamma}_{T}]=\hat{E}[\hat{\gamma}_{T}]\approx 8.563436$.
    In \cite[p.34]{Miura_Muller_Rupp_2025_optimal} the numerical value of an energy of the elastic two-teardrop is written as $C_{2T}\approx146.628$, but the correct value seems $C_{2T}\approx146.664860$. In fact, using the current notation, we have $C_{2T}=2L[\gamma_{2T}]B[\gamma_{2T}]=8L[\gamma_{T}]B[\gamma_{T}]=2\hat{E}[\hat{\gamma}_{T}]^2$.
\end{remark}

Now we formulate and solve the key variational problem to produce an optimal threshold.
We say that a complete curve $\gamma\colon\R\to\R^n$ has a \emph{tangential self-intersection} if there are $x_1,x_2\in\R$ such that $\gamma(x_1)=\gamma(x_2)$ and $\partial_s\gamma(x_1)\in\{\partial_s\gamma(x_2),-\partial_s\gamma(x_2)\}$.

\begin{definition}\label{def:elastic_pendant}
    We define the \emph{elastic pendant} $\gamma_P\in \dot{C}^{1,1}(\R;\R^2)$ as the complete properly immersed planar curve given by 
    \begin{align}
        \gamma_P(s):=
        \begin{cases}
            -\gamma_b^{\pi/2}(-s), & s\leq0,\\
            \hat{\gamma}_T(s)-\hat{\gamma}_T(0), & 0\leq s \leq L[\hat{\gamma}_T], \\
            \begin{pmatrix}
                1 & 0 \\
                0 & -1\\
            \end{pmatrix} \gamma_b^{\pi/2}(s-L[\hat{\gamma}_T]), & s \geq L[\hat{\gamma}_T],
        \end{cases}
    \end{align}
    see \Cref{fig:elastic_pendant}.
    Note that $N[\gamma_P]=0$.
\end{definition}

\begin{remark}
    In fact, the curve $\gamma_P$ has $\dot{C}^{2,1}$-regularity. This can be verified by direction computations as follows.
    By our choice of the parametrization it is clear that the sign of the curvature agrees at the joined points, so we only need to compute the absolute curvature at the endpoint of $\gamma_b^{\pi/2}$ and of $\hat{\gamma}_{T}$.
    By \eqref{eq:26-08}, the curvature of $\gamma_b^{\pi/2}$ at $s=0$ is $\sqrt{2}$.
    On the other hand, the absolute curvature of $\gamma_{T}$ in \eqref{eq:teardrop_before_rescaling} at the endpoints is $|2\sqrt{m_T}\cos(\pi-\alpha(m_T))|=\sqrt{4m_T-2}$ (as already computed in the proof of \cite[Lemma 2.28]{Miura_Muller_Rupp_2025_optimal}).
    Hence, the absolute curvature of $\hat{\gamma}_{T}$, defined as a reparametrization of $\sqrt{2m_T-1}\gamma_{T}$, is also given by $\frac{\sqrt{4m_T-2}}{\sqrt{2m_T-1}}=\sqrt{2}$ at the endpoints. The Lipschitz continuity of the curvature through the gluing point follows from the $W^{3,\infty}$-regularity of $\gamma_{T}$, see \cite[Lemma 2.12]{Miura_Muller_Rupp_2025_optimal}, and the fact that, by \eqref{eq:borderline}, we clearly have $\gamma_b\in W^{3,\infty}(\R;\R^2)$.
\end{remark}

\begin{theorem}\label{thm:elastic_pendant_minimality}
    Let $\gamma\in W^{2,2}_{loc}(\R;\R^{2})$ with $\inf_\R|\partial_x\gamma|>0$.
    If $N[\gamma]=0$ and $\gamma$ has a tangential self-intersection, then
    \begin{align}
        E[\gamma]\geq E[\gamma_P],
    \end{align}
    where equality holds if and only if $\gamma$ coincides with the elastic pendant $\gamma_P$ up to orientation-preserving reparametrization and direction-preserving isometry.
\end{theorem}

\begin{remark}\label{rem:pendant>8}
    It immediately follows that 
    \begin{align}E[\gamma_P]=E[\gamma_S]+E[\hat{\gamma}_{T}]=8-4\sqrt{2}+E[\hat{\gamma}_{T}].
    \end{align}
    Numerically, $E[\gamma_P]\approx 10.906581$.
    We can also rigorously verify that 
    $$E[\gamma_P]>E[\gamma_b]=8$$ 
    by a variational argument.
    By horizontally flipping the teardrop part of $\gamma_P$, we can create a curve $\hat{\gamma}_P$ with $E[\gamma_P]=E[\hat{\gamma}_P]$ and with a leftward tangent at the vertex of the teardrop.
    By \Cref{cor:serpent_general} we have $E[\hat{\gamma}_P]\geq8$, and equality is clearly not attained.
    Hence $E[\gamma_P]=E[\hat{\gamma}_P]>8$.
\end{remark}

\begin{proof}[Proof of \Cref{thm:elastic_pendant_minimality}]
    Suppose that $E[\gamma]<\infty$. By assumption there are $x_1,x_2\in\R$, $x_1<x_2$, such that $\gamma(x_1)=\gamma(x_2)$ and such that either $\partial_s\gamma(x_1)=\partial_s\gamma(x_2)$ (same direction) or $\partial_s\gamma(x_1)=-\partial_s\gamma(x_2)$ (opposite direction).
    We split the curve into three parts: $\gamma^1 \vcentcolon = \gamma \vert_{(-\infty,x_1]}$, $\gamma^2 \vcentcolon = \gamma\vert_{[x_1,x_2]}$, $\gamma^3\vcentcolon= \gamma\vert_{[x_2,\infty)}$ and consider these two cases separately.
    
    \emph{Case 1: Opposite direction.} Let $\varphi_i\in [0,\pi]$ denote the angle between $\partial_s\gamma(x_i)$ and $e_1$ for $i=1,2$. By assumption $\varphi_1 + \varphi_2 =\pi$. By \Cref{thm:initialangle_borderline_minimality}, we have 
    \begin{align}
        E[\gamma^1]+E[\gamma^3] \geq 8\sin^2 \frac{\varphi_1}{4} + 8 \sin^2 \frac{\varphi_2}{4}.
    \end{align}
    The unique minimizer over all $\varphi_1, \varphi_2\in[0,\pi]$ with $\varphi_1+\varphi_2=\pi$ is $\varphi_1=\varphi_2=\frac{\pi}{2}$. For $\gamma^2$, we may apply \Cref{thm:teardrop_minimality} together with \Cref{rem:teardrop_minimality} to conclude ${E}[\gamma^2] \geq {E}[\hat{\gamma}_{T}]$. %Since $\gamma^2$ closes up continuously, we have $\hat E[\gamma^2] = E[\gamma^2]$ and thus, 
    In total, we thus have
    \begin{align}
        E[\gamma] = E[\gamma^1]+E[\gamma^2]+E[\gamma^3]&\geq \left.8\sin^2 \frac{\varphi_1}{4} + 8 \sin^2 \frac{\varphi_2}{4}\right\vert_{\varphi_1=\varphi_2=\frac{\pi}{2}} + E[\hat{\gamma}_{T}] \\
        &= 8-4\sqrt 2 + E[\hat{\gamma}_{T}] = E[\gamma_P].
    \end{align}
    If equality is attained, then $\varphi_1=\varphi_2=\frac{\pi}{2}$ and, applying \Cref{thm:initialangle_borderline_minimality} to $x \mapsto -\gamma^1(x_1-x)$ and to $x\mapsto\gamma^3(x-x_2)$, we deduce that $\gamma^1$ coincides with $-\gamma_b^{\varphi_1}$ up to orientation-reversing reparametrization and direction-preserving isometry, and that $\gamma^3$ coincides with $\gamma_b^{\varphi_2}$ up to orientation-preserving reparametrization and direction-preserving isometry. Moreover, by \Cref{thm:teardrop_minimality}, $\gamma^2$ coincides with $\hat{\gamma}_{T}$ up to reparametrization and isometry. Finally, the assumption $N[\gamma]=0$ determines the way of parametrization. The statement then follows.
   
    \emph{Case 2: Same direction.} In this case, $\gamma^2$ is a $C^1$-closed curve and $\gamma^1$ and $\gamma^3$ can be joined to get a $C^1$-complete curve $\hat{\gamma}$. Clearly $N[\gamma] = N[\hat{\gamma}]+N[\gamma^2]$. Note that the latter is sum of integers (see for instance \cite[Definition A.3]{MR4565935} and \cite[Appendix B.2]{miura2025newenergymethodshortening}).
    We consider two cases, and in each case we prove that the minimum cannot be attained.
    
    Suppose first that $N[\hat{\gamma}]=0$.
    Then $N[\gamma^2]= 0$. By Hopf's Umlaufsatz \cite[Theorem A.5]{MR4565935}, $\gamma^2$ cannot be embedded, so by \cite[Theorem 1.2]{MR4565935} and \Cref{lem:energy_C0-closed}, we conclude that
    \begin{align}
        E[\gamma^2] \geq  E[\gamma_8],
    \end{align}
    where $\gamma_8$ is the figure-eight elastica of length $L[\gamma^2]=D[\gamma^2]$. By \Cref{lem:06-08} below, it follows that
    \begin{align}
        E[\gamma]\geq E[\gamma^2]\geq E[\gamma_8]> E[\gamma_P].
    \end{align}
    
    On the other hand, if $N[\hat{\gamma}]\neq 0$, then $\hat \gamma$ has a leftward tangent at some point. By \Cref{cor:serpent_general} with $\varphi_0=\pi$ and \eqref{eq:energy_8_borderline}, we have
    \begin{align}
        E[\hat\gamma]\geq 8 = E[\gamma_b].
    \end{align}
    Since circles are the global length-constrained minimizers of $B$, using \Cref{lem:energy_C0-closed}, we deduce that
    \begin{align}
        E[\gamma^2]\geq  E[\gamma_C],
    \end{align}
    where $\gamma_C$ is the circle of length $L[\gamma^2]=D[\gamma^2]$. Using \Cref{lem:06-08} below, we conclude
    \begin{align}
        E[\gamma] = E[\hat\gamma]+E[\gamma^2] \geq E[\gamma_b]+E[\gamma_C] > E[\gamma_P]. &\qedhere
    \end{align}
\end{proof}

\begin{lemma}\label{lem:06-08}
    For any figure-eight elastica $\gamma_8$ and any circle $\gamma_C$, the inequalities $E[\gamma_8]>E[\gamma_P]$ and $E[\gamma_b]+E[\gamma_C]> E[\gamma_P]$ hold.
\end{lemma}

Of course, the energies of these explicit shapes may also be estimated numerically. Recall that $E[\gamma_P]\approx 10.906581$. On the other hand, using \eqref{eq:0417-01} and the scaling invariance of $BL$, we have for any figure-eight elastica
\begin{align}
    E[\gamma_8] \geq 2\sqrt{B[\gamma_8]L[\gamma_8]} \approx 14.995973.
\end{align}
Similarly, for any circle, we obtain $E[\gamma_C] \geq 2\sqrt{2}\pi$, yielding
\[
E[\gamma_b]+E[\gamma_C] \geq 8+2\sqrt{2}\pi \approx 16.885766.
\]
However, in the following, we present a rigorous and geometric proof based on a cut-and-paste procedure (see \Cref{fig:cut-and-paste}) using the characterizations of $\gamma_b$ and $\hat{\gamma}_{T}$ as minimizers, cf.\ \Cref{thm:initialangle_borderline_minimality,thm:teardrop_minimality}.

\begin{figure}[htbp]
    \centering
    \includegraphics[width=0.9\linewidth]{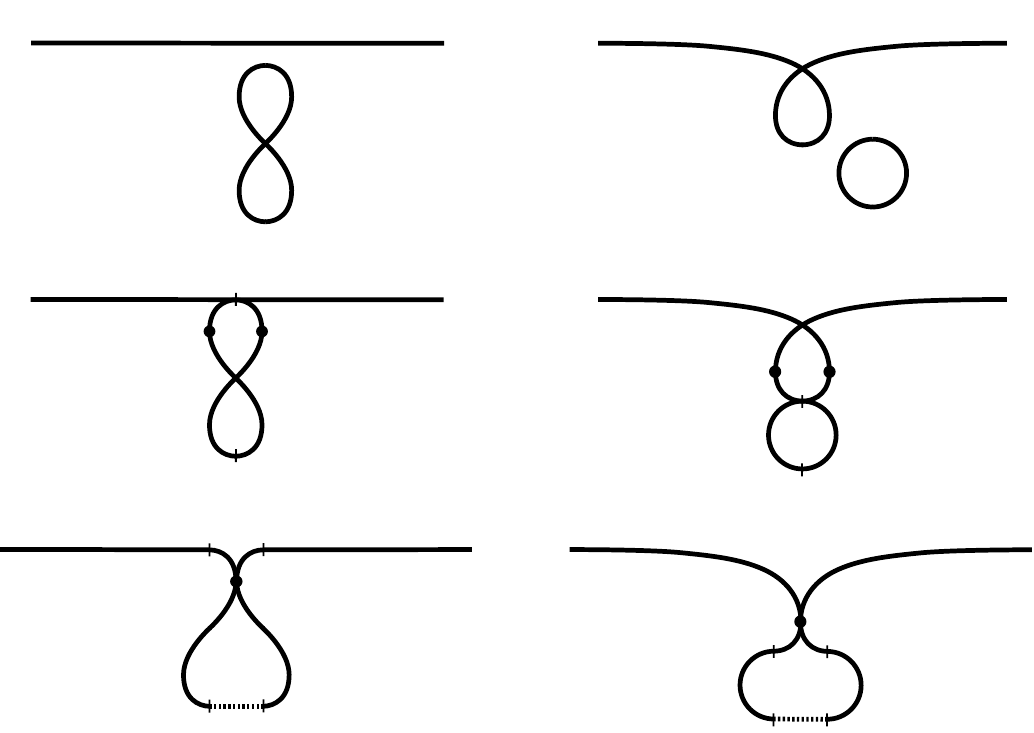}
    \caption{Cut-and-paste procedure (top to bottom). The vertical bars correspond to cutting points that are separated after translation. The dots mark points that are joined tangentially after translation.    
    The dashed lines represent the additional segments added in the process.}
    \label{fig:cut-and-paste}
\end{figure}

\begin{proof}[Proof of \Cref{lem:06-08}]
    For the first statement, let $\hat\gamma$ be the $e_1$-axis. After rotation and translation, we may assume that $\gamma_8$ intersects $\hat\gamma$ at the origin, lies entirely in the lower halfplane, and is axially symmetric with respect to the $e_2$-axis. Cutting $\gamma_8$ and $\hat\gamma$ at suitable points, translating the corresponding parts, and inserting a segment parallel to the $e_1$-axis, we obtain a complete curve having a tangential self-intersection with tangents $\pm e_2$, see the left column of \Cref{fig:cut-and-paste}. Denote by $\gamma^{\pm\infty}$ the non-compact arcs and by $\gamma^0$ the compact $C^0$-closed curve with a cusp.    
    By invariance of $B$ and $D$ and since segments in $e_1$-direction do not contribute, the adapted elastic energy has not changed, i.e.,
    \begin{align}
        E[\gamma^0]+E[\gamma^{+\infty}]+E[\gamma^{-\infty}]=E[\gamma_8].    
    \end{align}    
    By \Cref{thm:initialangle_borderline_minimality} with $\varphi=\frac{\pi}{2}$, we have $E[\gamma^{+\infty}]+E[\gamma^{-\infty}] \geq E[\gamma_S]$. On the other hand, applying \Cref{thm:teardrop_minimality} together with \Cref{rem:teardrop_minimality} to the compact part, we find $ E[\gamma^0]>  E[\hat{\gamma}_{T}]$. 
    
    For the second statement, we may proceed similarly. After rotation and translation, and 
    after cutting the circle and the borderline elastica at suitable points and inserting a segment parallel to the $e_1$-axis, we again end up with a complete curve having a tangential self-intersection with tangents $\pm e_2$ with the same adapted elastic energy, see the right column of \Cref{fig:cut-and-paste}. Proceeding exactly as above, the statement follows.
\end{proof}

We can now prove \Cref{thm:EF_embedded}.

\begin{proof}[Proof of \Cref{thm:EF_embedded}]
    Since the only embedded stationary solution with finite energy is the line by \cite[Theorem 7.11 (i)]{miura2025newenergymethodshortening}, as in the proof of \Cref{thm:EF_graphical}, we may assume that $E[\gamma(t, \cdot)]<E[\gamma_P]$ for all $t>0$. We define
    \begin{align}
        t_* \vcentcolon= \sup\{ \tau>0 : \gamma(t,\cdot) \text{ is embedded for all }t\in [0,\tau]\}.
    \end{align}
    Suppose $t_*<\infty$. Then $\gamma(t_*,\cdot)$ has a self-intersection which, by continuity and stability of transversal self-intersections, must be tangential. Since the rotation number $N$ is preserved and zero for embedded curves, we conclude from \Cref{thm:elastic_pendant_minimality} that $E[\gamma(t_*,\cdot)]\geq E[\gamma_P]$, a contradiction. It follows that $t_*=\infty$, so embeddedness is preserved.

    For the classification of the limit, we note that the borderline elastica $\gamma_b$ cannot arise as a subsequential limit $\gamma_\infty$ in \cite[Theorems 7.11 and 7.12]{miura2025newenergymethodshortening}, since $\gamma_b$ has a transversal self-intersection which is stable; any such limit $\gamma_\infty$ is thus a straight line. Arguing as in the proof of \cite[Theorem 2.2]{miura2025newenergymethodshortening}, the convergence statement in \Cref{thm:EF_embedded} follows.

    Lastly, to prove optimality of the energy threshold, we perturb $\gamma_P$ by replacing a sufficiently small neighborhood of the tangential self-intersection with the graphs of $\pm (x^4+\alpha)$ as in \cite[Section 3.2]{Miura_Muller_Rupp_2025_optimal}. The breaking of embeddedness can then be shown as in \cite[Proposition 3.4]{Miura_Muller_Rupp_2025_optimal} using well-posedness of the flow equation, see for instance \cite[Theorems A.2 and A.3]{miura2025newenergymethodshortening}.
\end{proof}

\section{Li--Yau type inequalities for complete curves}\label{sec:Li-Yau}

In this section we obtain Li--Yau type inequalities for complete planar curves, including a new variational characterization of the borderline elastica (\Cref{thm:borderline_minimality}), and also discuss some difficulties in higher codimensions.

A key step is to improve \Cref{thm:elastic_pendant_minimality}.
While sufficient for our application to the elastic flow, in fact, we can improve it by removing the assumption that the self-intersection is tangential, so that the statement becomes more analogous to the case of closed curves \cite{Miura_Muller_Rupp_2025_optimal}.

Note that here we focus on the planar case, so direction-preserving isometries are only given by translations, vertical reflections, and their compositions.

\begin{theorem}\label{thm:elastic_pendant_minimality_without_tangentiality}
    Let $\gamma\in W^{2,2}_{loc}(\R;\R^2)$ with $\inf_\R|\partial_x\gamma|>0$.
    If $N[\gamma]=0$ and $\gamma$ has a self-intersection, then
    \begin{align}
        E[\gamma]\geq E[\gamma_P],
    \end{align}
    where equality holds if and only if $\gamma$ coincides with the elastic pendant $\gamma_P$ up to orientation-preserving reparametrization and direction-preserving isometry.
\end{theorem}

We will prove this by using the direct method.
However, since admissible curves have infinite length, standard compactness arguments can only ensure local-in-space convergence along a minimizing sequence, and in particular the rotation number may not be preserved in the limiting process.
To ensure that the limit curve has rotation number zero, we need to precisely estimate not only the energy of the limit curve but also of the minimizing sequence itself.

To this end we first obtain a sharp lower bound of the adapted elastic energy for planar curves in terms of the total curvature (or equivalently, the variation of tangential angle) in the spirit of \cite{Miura20}.
Let $\lfloor \cdot \rfloor$ denote the floor function.

\begin{lemma}\label{lem:lower_bound_of_E_by_N}
    Let $I\subset\R$ be an interval, possibly unbounded, and $\gamma\in W^{2,2}_{loc}(I;\R^2)$ be an arclength parametrized planar curve.
    Suppose that $\Delta:=\left|\int_I k ds\right|$ is well defined in the sense of improper integrals.
    Then
    \begin{align}
        E[\gamma] \geq 8 \left\lfloor \frac{\Delta}{2\pi} \right\rfloor +  16\sin^2\left(\frac{\Delta-2\pi \left\lfloor \frac{\Delta}{2\pi} \right\rfloor}{8}\right).
    \end{align}
    Equality holds if and only if either $\gamma$ is a line segment in the direction of $e_1$ (where $\Delta=0$ and $I$ is arbitrary), or $\gamma$ coincides with $\gamma_b$ up to parameter shift and direction-preserving isometry (where $\Delta= 2\pi$ and $I=\R$).
\end{lemma}

\begin{proof}
    Up to vertical reflection we may assume that $\Delta=\int_Ikds\geq0$ throughout the proof.
    
    Let $a,b\in[-\infty,\infty]$ denote the endpoints of the domain interval $I$, where $a<b$.
    Let $\theta:I\to\R$ be a tangential angle function such that $\partial_s\gamma=(\cos\theta,\sin\theta)$.
    By assumption the limits $\theta(a):=\lim_{x\to a+}\theta(x)$ and $\theta(b):=\lim_{x\to b-}\theta(x)$ are well defined.
    Then we have $\Delta=\theta(b)-\theta(a)\geq0$.
    
    Now we estimate $E[\gamma]$.
    The Cauchy--Schwarz inequality yields
    \begin{align}\label{eq:0813-10}
        E[\gamma]&=\int_I \left(\frac{1}{2}|\partial_s\theta|^2+1-\cos\theta\right)ds \geq \int_I 2\cdot\frac{1}{\sqrt{2}}|\partial_s\theta|\cdot\sqrt{1-\cos\theta} ds\\
        &\geq \left|\int_I \partial_s\theta\sqrt{2-2\cos\theta} ds \right|= \int_{\theta(a)}^{\theta(b)} \sqrt{2-2\cos\theta} d\theta .
    \end{align}
    Now the last term is an integral of an explicit nonnegative periodic function with period $2\pi$, over an integral interval of width $\Delta=\theta(b)-\theta(a)$.
    Since the integral over one period is $\int_0^{2\pi}\sqrt{2-2\cos\theta} d\theta=8$, we obtain
    \begin{align}
        \int_{\theta(a)}^{\theta(b)} \sqrt{2-2\cos\theta} d\theta = 8 \left\lfloor \frac{\Delta}{2\pi} \right\rfloor +  \int_{I'} \sqrt{2-2\cos\theta} d\theta,
    \end{align}
    where $I'$ is an interval of width $|I'|=\Delta-2\pi \left\lfloor \frac{\Delta}{2\pi} \right\rfloor \in [0,2\pi)$.
    The last integral is minimized when $I'$ has the midpoint at the origin (or on $2\pi\Z$), thus bounded by
    \begin{align}\label{eq:0813-11}
        \int_{I'} \sqrt{2-2\cos\theta} d\theta \geq \int_{-|I'|/2}^{|I'|/2} \sqrt{2-2\cos\theta} d\theta = 16\sin^2\frac{|I'|}{8}.
    \end{align}
    Combining the above estimates and inserting $|I'|=\Delta-2\pi \left\lfloor \frac{\Delta}{2\pi} \right\rfloor$ complete the proof of the inequality.

    Finally we consider the case of equality.
    Equality is attained in \eqref{eq:0813-10} if and only if we have either $\partial_s\theta=\sqrt{2-2\cos\theta}$ or $\partial_s\theta=-\sqrt{2-2\cos\theta}$ on the whole interval $I$; our assumption $\Delta\geq0$ then implies $\partial_s\theta=\sqrt{2-2\cos\theta}$.
    Solving the ODE shows that $\gamma$ coincides with either the line $s\mapsto s e_1$ or the borderline elastica $\gamma_b$, or part thereof, up to translation and parameter shift, cf.\ \cite[(3.5)--(3.6)]{Miura20}.
    In the line case it is clear that equality holds, so hereafter we may only consider the borderline case.
    In addition, equality in \eqref{eq:0813-11} is attained if and only if either $|I'|=0$ or $I'$ is a non-degenerate interval with midpoint on $2\pi\Z$ (at which the curve $\gamma$ has a rightward tangent).
    However, under equality in \eqref{eq:0813-10}, the latter case cannot occur since $\gamma_b$ does not have a rightward tangent, and hence only the former can occur, yielding $\Delta\in2\pi\Z$, i.e., $N[\gamma]\in\Z$.
    Then, since $N[\gamma_b]=1$ and $k_b>0$, equality in \eqref{eq:0813-10} and in \eqref{eq:0813-11} both  hold if and only if $I=\R$, $\Delta =2\pi$, and $\gamma$ exactly coincides with $\gamma_b$ up to translation and parameter shift.
\end{proof}

In particular, if we assume in advance that $I=\R$ and the rotation number $N[\gamma]$ is well-defined as an integer, then $\Delta=2\pi|N[\gamma]|\in 2\pi\Z$ and hence we obtain the following sharp energy lower bound in terms of the rotation number.

\begin{corollary}\label{cor:lower_bound_of_E_by_N}
    Let $\gamma\in W^{2,2}_{loc}(\R;\R^2)$ with $\inf_\R|\partial_x\gamma|>0$.
    Then
    \begin{align}
        E[\gamma] \geq 8 |N[\gamma]|,
    \end{align}
    whenever both sides exist. 
    Equality holds if and only if either $\gamma$ is a line in the direction of $e_1$ (where $N[\gamma]=0$), or $\gamma$ coincides with $\gamma_b$ up to orientation-preserving reparametrization and direction-preserving isometry (where $|N[\gamma]|=1$).
\end{corollary}

\begin{proof}
    We may assume that $\gamma$ is parametrized by arclength, and further $E[\gamma]<\infty$, otherwise there is nothing to show. In this case, $N[\gamma]\in \Z$, see \cite[Appendix B.2]{miura2025newenergymethodshortening}. The statement then follows from \Cref{lem:lower_bound_of_E_by_N}.
\end{proof}

\begin{remark}
    \Cref{cor:lower_bound_of_E_by_N} is not only optimal for $|N|\leq 1$, but also sharp for all $|N|\geq2$ in the sense that the lower bound cannot be improved.
    This fact can easily be shown by constructing a sequence of $|N|$ diverging borderline elasticae, suitably glued together.
\end{remark}

We are now in a position to prove \Cref{thm:elastic_pendant_minimality_without_tangentiality}.

\begin{proof}[Proof of \Cref{thm:elastic_pendant_minimality_without_tangentiality}]
    We prove existence and uniqueness of minimizers of $E$ among all admissible curves.

    \emph{Step 1: Existence.} 
    Take a minimizing sequence $\{\gamma_j\}\subset W^{2,2}_{loc}(\R;\R^2)$, that is, $\lim_{j\to\infty}E[\gamma_j]=\inf E$, where the infimum runs over all admissible curves.
    Note that $\inf E \leq E[\gamma_P]$ holds since $\gamma_P$ is admissible, while $E[\gamma_P]<16$ follows by a simple energy comparison (see \Cref{rem:pendant<16} below).
    Hence
    \begin{align}\label{eq:0811-03}
        \lim_{j\to\infty}E[\gamma_j] < 16.
    \end{align}
    
    After taking orientation-preserving reparametrizations and translations, we may assume that all curves in $\{\gamma_j\}$ are parametrized by arclength, and there is $\{s_j\}\subset(0,\infty)$ such that
    \begin{equation}\label{eq:self-intersection_j}
        \gamma_j(0)=\gamma_j(s_j)=0.
    \end{equation}

    We now show that the sequence $\{s_j\}$ is bounded above and away from zero.
    Using \eqref{eq:self-intersection_j} and \Cref{lem:energy_C0-closed}, we find $\sup_js_j=\sup_jL[\gamma_j|_{[0,s_j]}]=\sup_jD[\gamma_j|_{[0,s_j]}]\leq \sup_j{E}[\gamma_j|_{[0,s_j]}]\leq \sup_j E[\gamma_j] <\infty$, while the lower estimate $\inf_j s_j>0$ follows by the boundedness of the bending energy $B[\gamma_j|_{[0,s_j]}]\leq {E}[\gamma_j|_{[0,s_j]}]$ together with the estimate
    \begin{align}
        s_jB[\gamma_j|_{[0,s_j]}]=L[\gamma_j|_{[0,s_j]}]B[\gamma_j|_{[0,s_j]}]\geq \Big(\int_{\gamma_j|_{[0,s_j]}}|\kappa|ds\Big)^2\geq \pi^2,
    \end{align}
    where we used the Cauchy--Schwarz inequality and a version of Fenchel's theorem; more precisely, since $\gamma_j(0)=\gamma_j(s_j)$, we can create a closed curve by concatenating $\gamma_j$ and its point reflection at the endpoint, to which Fenchel's theorem applies.
    
    Hence, up to taking a subsequence (without relabeling), we may assume the existence of the limit $\lim_{j\to\infty}s_j = s_\infty\in(0,\infty)$.
    Now a standard compactness and diagonal argument ensure that a subsequence of $\{\gamma_j\}$ (without relabeling) locally converges to an arclength parametrized curve $\gamma_\infty\in W^{2,2}_{loc}(\R;\R^2)$ in the $W^{2,2}$-weak and $C^1$-strong sense.
    In addition, taking $j\to\infty$ in \eqref{eq:self-intersection_j} we also have $\gamma_\infty(0)=\gamma_\infty(s_\infty)=0$ for $s_\infty\neq0$, which ensures the existence of a self-intersection.
    By lower semicontinuity and \eqref{eq:0811-03}, we have
    \[
    E[\gamma_\infty]\leq\lim_{j\to\infty}E[\gamma_j] < 16.
    \]
    Hence $N[\gamma_\infty]\in\Z$ and moreover  \Cref{cor:lower_bound_of_E_by_N} implies $N[\gamma_\infty]\in\{-1,0,1\}$.
    For ensuring the admissibility of $\gamma_\infty$, the remaining task is to show that $|N[\gamma_\infty]|\neq1$.
    
    Suppose on the contrary that $|N[\gamma_\infty]|=1$; up to reflection we may assume $N[\gamma_\infty]=1$.
    Then \Cref{cor:lower_bound_of_E_by_N} yields $E[\gamma_\infty]\geq8$ and hence for any small $\varepsilon>0$ there is a large $R_\varepsilon>0$ such that
    \begin{align}\label{eq:0811-01}
        \liminf_{j\to\infty}E[\gamma_j|_{[-R_\varepsilon,R_\varepsilon]}] \geq  E[\gamma_\infty|_{[-R_\varepsilon,R_\varepsilon]}] \geq 8 - \varepsilon.
    \end{align}
    In addition, combining $N[\gamma_\infty]=1$ with \Cref{prop:finite_E_horizontal_end} implies that the tangential angle $\theta_\infty$ of $\gamma_\infty$ satisfies
    \begin{align}
        |\theta_\infty(-R_\varepsilon)|\leq \varepsilon, \quad |\theta_\infty(R_\varepsilon)-2\pi| \leq \varepsilon,
    \end{align}
    after taking larger $R_\varepsilon$ if necessary.
    Hence, there is a large $j_\varepsilon$ such that for all $j\geq j_\varepsilon$
    the tangential angle $\theta_j$ of $\gamma_j$ satisfies
    \begin{align}\label{eq:R_epsilon_angle}
        |\theta_j(-R_\varepsilon)| \leq 2\varepsilon, \quad |\theta_j(R_\varepsilon) - 2\pi| \leq 2\varepsilon.
    \end{align}
    Since $\lim_{s\to-\infty}\theta_j(s)=\lim_{s\to\infty}\theta_j(s)$ holds due to $N[\gamma_j]=0$, we find that for all $j\geq j_\varepsilon$, either $\lim_{s\to-\infty}\theta_j(s)\in 2\pi\Z\setminus\{0\}$ or $\lim_{s\to\infty}\theta_j(s)\in 2\pi\Z\setminus\{2\pi\}$ holds.
    This together with \eqref{eq:R_epsilon_angle} implies that
    \begin{align}
        \max\left\{|\theta_j(\infty)-\theta_j(R_\varepsilon)|,|\theta(-R_\varepsilon)-\theta(-\infty)| \right\} \geq 2\pi-2\varepsilon.
    \end{align}
    Hence, by \Cref{lem:lower_bound_of_E_by_N},
    \begin{align}\label{eq:0811-02}
        \max\{E[\gamma_j|_{[R_\varepsilon,\infty)}],E[\gamma_j|_{(-\infty,-R_\varepsilon]}]\} \geq 8-O(\varepsilon).
    \end{align}
    Therefore, combining \eqref{eq:0811-01} and \eqref{eq:0811-02}, for any $\varepsilon>0$ we obtain
    \begin{align}
        \lim_{j\to\infty}E[\gamma_j] &\geq \liminf_{j\to\infty}\Big(E[\gamma_j|_{[-R_\varepsilon,R_\varepsilon]}] + \max\{E[\gamma_j|_{[R_\varepsilon,\infty)}],E[\gamma_j|_{(-\infty,-R_\varepsilon]}]\}\Big)\\
        & \geq 16-O(\varepsilon).
    \end{align}
    Taking $\varepsilon\to0$ thus yields $\lim_{j\to\infty}E[\gamma_j] \geq 16$, which contradicts \eqref{eq:0811-03}.

    Thus we obtain $N[\gamma_\infty]=0$.
    Since $E[\gamma_\infty]\leq\lim_{j\to\infty}E[\gamma_j]=\inf E$ while $\gamma_\infty$ is admissible, this shows that $\gamma_\infty$ is a minimizer.

    \emph{Step 2: Uniqueness.}
    Let $\bar{\gamma}$ be any minimizer.
        Then $\bar{\gamma}$ must have a tangential self-intersection; indeed, if the self-intersection is transversal, then any local perturbation is admissible and hence the minimality of $\bar{\gamma}$ implies that $\bar{\gamma}$ is an elastica; however, since $E[\bar{\gamma}]<\infty$ and $N[\bar{\gamma}]=0$, this implies that $\bar{\gamma}$ is a line (again using the known classification of planar elasticae \cite{Singer_lecturenotes,Miura_elastica_survey}), which contradicts the fact that $\bar{\gamma}$ has a self-intersection.
    Thus we have shown that any minimizer $\bar{\gamma}$ with $N[\bar{\gamma}]=0$ has a tangential self-intersection.
    Therefore, \Cref{thm:elastic_pendant_minimality} implies the desired uniqueness up to invariances.
\end{proof}

\begin{remark}\label{rem:pendant<16}
    Here we rigorously verify the estimate $E[\gamma_P]<16$.
    By \Cref{lem:energy_C0-closed},
    \[
    E[\gamma_P] = E[\gamma_S] + E[\hat{\gamma}_{T}] = 8-4\sqrt{2} + \hat{E}[\hat{\gamma}_{T}].
    \]
    It suffices to show that $\hat{E}[\hat{\gamma}_{T}]<8+4\sqrt{2}$.
    This follows by the minimality of $\hat{\gamma}_{T}$ in \Cref{thm:teardrop_minimality}.
    For example, we deduce that $\hat{E}[\hat{\gamma}_{T}]<7\sqrt{2}\pi/3$ (which is clearly less than $8+4\sqrt{2}$), where the right hand side is the energy of the competitor constructed as follows.
    Let three circles of radius $1/\sqrt{2}$ be arranged so that they are mutually tangent and the points of tangency form the vertices $A$, $B$, and $C$ of an equilateral triangle. 
    Between each pair of vertices there are two circular arcs: a shorter one and a longer one. 
    The competitor is obtained by traversing the short arc from $A$ to $B$, the long arc from $B$ to $C$, and the short arc from $C$ back to $A$. 
    The resulting planar curve starts and ends at $A$ with opposite tangent directions.
    Since the curve consists of three circular arcs of radius $1/\sqrt{2}$ and angles $\pi/3$, $5\pi/3$, and $\pi/3$, the total length is $L=(1/\sqrt{2})(\pi/3+5\pi/3+\pi/3)=7\sqrt{2}\pi/6$ and the total energy is $\hat{E}=(\frac{1}{2}(\sqrt{2})^2+1) L = 7\sqrt{2}\pi/3$.
\end{remark}

Finally, applying the above results, we immediately obtain an unconditional Li--Yau type inequality for complete planar curves.

\begin{theorem}\label{thm:borderline_minimality}
    Let $\gamma\in W^{2,2}_{loc}(\R;\R^2)$ with $\inf_\R|\partial_x\gamma|>0$.
    If $\gamma$ has a self-intersection, then
    \begin{equation}
        E[\gamma] \geq 8,
    \end{equation}
    where equality holds if and only if $\gamma$ coincides with the borderline elastica $\gamma_b$ up to orientation-preserving reparametrization and direction-preserving isometry.
\end{theorem}

\begin{proof}
    We may assume that $E[\gamma]<\infty$.
    In this case $N[\gamma]\in\Z$.
    If $|N[\gamma]|\geq1$, then the inequality and the rigidity in the equality case follow from \Cref{cor:lower_bound_of_E_by_N} directly.
    On the other hand, if $N[\gamma]=0$, then \Cref{thm:elastic_pendant_minimality_without_tangentiality} together with \Cref{rem:pendant>8} implies that $E[\gamma]>8$.
    The proof is thus complete.
\end{proof}

The following statements are direct corollaries of \Cref{thm:borderline_minimality}.

\begin{corollary}
    Let $\gamma\in W^{2,2}_{loc}(\R;\R^2)$ with $\inf_\R|\partial_x\gamma|>0$.
    If $E[\gamma]<8$, then $\gamma$ is embedded.
\end{corollary}

\begin{corollary}
    Let $\gamma_0\in\dot{C}^\infty(\R;\R^2)$ with $\inf_\R|\partial_x\gamma|>0$.
    If $E[\gamma_0]\leq 8$, then the unique elastic flow $\gamma:[0,\infty)\times\R\to\R^2$ starting from $\gamma_0$ is either a stationary solution given by a borderline elastica, or $\gamma(t,\cdot)$ is embedded for all $t\geq0$.
\end{corollary}

In fact, we conjecture that \Cref{thm:borderline_minimality} remains true for all codimensions.

\begin{conjecture}\label{conj:boderline_minimality_high_codim}
    Let $n\geq2$ and $\gamma\in W^{2,2}_{loc}(\R;\R^n)$ with $\inf_\R|\partial_x\gamma|>0$.
    If $\gamma$ has a self-intersection, then
    \begin{equation}
        E[\gamma] \geq 8,
    \end{equation}
    where equality holds if and only if $\gamma$ coincides with the borderline elastica $\gamma_b$ up to orientation-preserving reparametrization and direction-preserving isometry.
\end{conjecture}

However, as opposed to the codimension one case, we cannot rely on the rotation number. Instead, we need to solve a variational problem with a constraint, involving a free boundary given by the self-intersections. 

In the closed curve case \cite{MR4631455}, a method of relaxing the free boundary condition was successful. Indeed, any closed curve with a self-intersection may be split into two loops and, since they arise from a closed curve, the tangents at the endpoints of the loops must align suitably. Minimizing the elastic energy thus naturally comes with $C^1$-type boundary conditions relating the two loops. Relaxing these to $C^0$-boundary conditions makes the problem analytically much simpler, since each loop can be minimized independently. Remarkably, by explicitly identifying the minimizer of the relaxed problem, the figure-eight elastica, both variational problems were shown to have the same minimizer in \cite{MR4631455}. 

This approach cannot be applied in the non-compact regime. Indeed, splitting a complete curve at a self-intersection yields one loop and two semi-infinite arcs. However, for example, the minimal energy $E$ among semi-infinite arcs is zero (attained by semi-lines) which does not correspond to the optimal shape $\gamma_b$ in \Cref{thm:borderline_minimality}. Hence, as opposed to \cite{MR4631455}, minimizing $E$ independently on all three parts necessarily results in a non-optimal threshold.

\bibliography{Lib}
\end{document}